\documentclass{aims}
\usepackage{amsmath}
  \usepackage{paralist}
  \usepackage{graphics} 
  \usepackage{epsfig} 
 \usepackage[colorlinks=true]{hyperref}

\allowdisplaybreaks[4]

\usepackage{graphicx,amssymb,mathrsfs,latexsym,amsthm}
\usepackage{verbatim}

\usepackage{color,soul,framed,xcolor}

\definecolor{shadecolor}{rgb}{1,1,0}

\hypersetup{urlcolor=blue, citecolor=red}

    \def\ppages{X--XX}

\newtheorem{theorem}{Theorem}[section]
\newtheorem{corollary}{Corollary}[section]

\newtheorem{lemma}{Lemma}[section]
\newtheorem{proposition}{Proposition}[section]

\theoremstyle{definition}
\newtheorem{definition}{Definition}[section]
\newtheorem{remark}{Remark}[section]

\title[Mean Topological Dimension]
      {Mean Topological Dimension for random bundle transformations}

\author[Junqi Yang, Xianfeng Ma and Ercai Chen]{}

 \email{y30140124@mail.ecust.edu.cn}
 \email{xianfengma@gmail.com}
 \email{ecchen@njnu.edu.cn}

\thanks{The second and third authors are  supported by  NNSF of China (11271191).
The  second author is supported by  NNSF of China (11471114).
The third author is partially supported by National
Basic Research Program of China (973 Program) (Grant No. 2013CB834100).
}

\begin{document}
\maketitle

\renewcommand{\thefootnote}{\fnsymbol{footnote}}
\centerline{
\scshape Junqi Yang
}
\medskip
{\footnotesize
   \centerline{Department of Mathematics}
   \centerline{East China University of Science and Technology, Shanghai 200237, China}
} 

\medskip

\centerline{
\scshape Xianfeng Ma \footnote{Corresponding author}
}
\medskip
{\footnotesize
   \centerline{Department of Mathematics}
   \centerline{East China University of Science and Technology, Shanghai 200237, China}
} 

\medskip

\centerline{\scshape Ercai Chen}
\medskip
{\footnotesize
 \centerline{School of Mathematical Science}
   \centerline{Nanjing Normal University, Nanjing 210097, China}
   \centerline{and}
   \centerline{Center of Nonlinear Science}
   \centerline{Nanjing University, Nanjing 210093, China}
} %

\bigskip


\begin{abstract}We introduce the mean topological dimension for random bundle transformations, and show that continuous bundle random dynamical systems with finite topological entropy, or the small boundary property have zero mean topological dimensions.
\end{abstract}

\section{Introduction}\label{sec1}
Topological entropy plays an important role in the theory of dynamical systems.
It measures the exponential growth rate of the number of distinguishable orbits of the iterates and represents the complexity of dynamical systems.
It was first introduced by Adler, Konheim and McAndrew \cite{AKM} as an invariant of topological conjugacy for studying dynamical systems in compact topological spaces.
An equivalent definition was introduced in metric spaces by Bowen \cite{Bowen1} and Dinaburg \cite{Dinaburg} independently.

Shub and Weiss \cite{shub1991can} developed the notion of small sets that plays  a crucial role in the study of small entropy factors of  topological dynamical systems.
Later, mean topological dimension, an analogue of Lebesgue covering dimension, was introduced by Gromov \cite{gromov1999topological} for studying dynamical properties of certain spaces of holomorphic maps and complex varieties.
 Lindenstrauss and Weiss \cite{lindenstrauss2000mean} presented the notion of the mean dimension of  dynamical systems and used it to answer in the negative an open question raised by Auslander \cite{auslander1988minimal} that whether every minimal system $(X,T)$ can be imbedded in  $[0,1]^{\mathbb{Z}}$.
They also defined the metric mean dimension which can be thought as a mean Minkowski dimension by using the cardinality of open covers by sets of small diameter instead of the degree of these covers.
It allowed them to establish the relationship between the mean dimension and the topological entropy of dynamical systems, which indicates that each system with finite topological entropy has zero mean dimension.
By replacing the empty set with the small set in the definition of zero dimensional space, they also developed the small boundary property (SBP), which could be seen as a dynamical version of being totally disconnected, and showed that a dynamical system with this property has mean dimension zero. The mean topological dimension can give some interesting information in deterministic dynamical systems especially when the usual invariant of topological entropy is infinite
\cite{Lin1999,lindenstrauss2000mean}.

Random dynamical systems (RDS) evolve by the composition of different maps instead of  iterations of one self-map.
 Kifer \cite{kifer2012ergodic} studied the  systems generated by random transformations chosen independently  with identical distribution and introduced the notions of topological entropy and topological pressure.
Bogensch{\"u}tz \cite{Bogen} gave the definition of the topological entropy for the  random transformations acting on one space.
A more general models of random transformations are formed by skew-product maps restricted to random invariant sets and act between different spaces as a class of bundle RDS.
 Kifer \cite{kifer2001topological} showed that in this situation the corresponding topological pressure can be obtained by almost sure limits and deduced the relativized variational principle.

In the present paper, we follow the work from \cite{lindenstrauss2000mean,kifer2001topological} and introduce the mean topological dimension for continuous bundle RDS, which enables us to assign a  quantity to the systems with infinite dimensional state space (for example, generated by infinite dimensional stochastic differential equations) or infinite topological entropy.
We also define the metric mean dimension and the SBP for bundle RDS, which include the deterministic case as the probability measure being supported on a single point.
We show  that  random dynamical systems with finite topological entropy or the SBP have zero mean topological dimensions.

The paper is organized as follows.
In Section \ref{sec2}, we recall some background in bundle RDS and prove the measurability of the function $ \mathcal{D}\big(\alpha^{n-1}_0\left( \omega \middle\vert \mathcal{E} \right)\big)$ in $\omega$.
In Section \ref{sec3}, we define the mean topological dimension and the metric mean dimension in bundle RDS, investigate the relationship between the two notions,  and show that the finite topological entropy implies the zero mean topological dimension.
In Section \ref{sec4}, we introduce the small boundary property for random bundle transformations and devote to proving that the SBP implies the zero mean topological dimension.

\section{Preliminaries and the measurability of  $ \mathcal{D}\big(\alpha^{n-1}_0\left( \omega \middle\vert \mathcal{E} \right)\big)$}\label{sec2}

Let $(\Omega, \mathcal{F},\mathbb{P})$ be a complete probability space together with a $\mathbb{P}$-preserving transformation $\vartheta$ and $X$ be a compact metric space with the distance function $d$ and the Borel $\sigma$-algebra $\mathcal{B}$.
Let $\mathcal{E}$ be a measurable subset of $\Omega\times X$ with respect to the product $\sigma$-algebra $\mathcal{F}\otimes \mathcal{B}$ and the fibers
$\mathcal{E}_{\omega}=\{x\in X: (\omega,x)\in \mathcal{E}\},\omega\in\Omega$ are nonempty compact subsets of $X$.
It means (see \cite{castaing2006convex}) that the mapping $\omega\mapsto\mathcal{E}_\omega$ is measurable with respect to the Borel $\sigma$-algebra induced by the Hausdorff topology $\mathscr{T}_H$ on the space $\mathscr{P}_{k}$ of compact subsets of $X$.
A continuous bundle random dynamical systems (RDS) $T$ over $(\Omega, \mathcal{F},\mathbb{P},\vartheta)$ is generated by the mappings $T_{\omega}:\mathcal{E}_{\omega}\rightarrow \mathcal{E}_{\vartheta\omega}$
so that the map $(\omega,x)\mapsto T_{\omega}x$ is measurable and the map $x\mapsto T_{\omega}x$ is continuous for $\mathbb{P}$-almost all (a.a.) $\omega$.
The family $\{T_{\omega}:\omega\in \Omega\}$ is called a random transformation and each $T_{\omega}$ maps the fiber $\mathcal{E}_{\omega}$ to $\mathcal{E}_{\vartheta\omega}$.
The map $\Theta:\mathcal{E}\rightarrow \mathcal{E}$ defined by $\Theta(\omega,x)=(\vartheta\omega, T_{\omega}x)$ is called the skew product transformation. Observe that $\Theta^n(\omega, x)=(\vartheta^n\omega, T_{\omega}^nx)$, where $T_{\omega}^n=T_{\vartheta^{n-1}\omega}\circ\cdots T_{\vartheta\omega}\circ T_{\omega}$ for $n\geq 1$ and $T_{\omega}^0=id$.

For each finite open cover $\alpha=\{A^{(i)}:i=1,2,\dots,l\}$ of $X$, denote by $\alpha\left(\omega\middle\vert\mathcal{E}\right)$ the open cover of $\mathcal{E}_\omega$ by the sets $A^{(i)}_{\mathcal{E}}(\omega)=A^{(i)}\cap\mathcal{E}_\omega$ (some of which may be empty).
Clearly, $\vee_{i=0}^{n-1} (T_{\omega}^{i})^{-1} \alpha\left(\omega\middle\vert\mathcal{E}\right)$ denoted by $\alpha_{0}^{n-1}\left(\omega\middle\vert\mathcal{E}\right)$ is the open cover of $\mathcal{E}_\omega$ consisting of sets \[A^{(j_0,j_1,\dots,j_{n-1})}_{\mathcal{E}}(\omega)=\bigcap_{i=0}^{n-1} (T^i_\omega)^{-1}A^{(j_i)}_{\mathcal{E}}(\theta^i \omega)\](some of which may be empty).
By the measurability of $\Theta$, it is easy to prove (see \cite{kifer2001topological}) that for any $j=(j_0,j_1,\dots,j_{n-1})\in {\{1,2,\dots,l\}}^{n}$, the graph of map $\omega \mapsto A^{j}_{\mathcal{E}}(\omega)$, denoted by $A^j_{\mathcal{E}}=\{(\omega,x):x\in A^{j}_{\mathcal{E}}(\omega)\}$ is measurable.

Let $K\subset X$ be compact, $\alpha$ a finite open cover of $K$. We shall denote (see \cite{lindenstrauss2000mean})
\[
\text{ord}(\alpha)=-1+\sup_{x\in K} \sum_{U\in \alpha} 1_{U}(x) \text{ and } \mathcal{D}(\alpha)=\min_{\beta\succ\alpha} \text{ord}(\beta)
\]
where $\beta$ runs over all finite open covers of $K$ refining $\alpha$. Since $K$ is normal, it can be proved (see \cite{coornaert2015topological}) that in the definition of $\mathcal{D}(\alpha)$, one can make $\beta$ runs over all finite closed covers of $K$ that are finer than $\alpha$.

\begin{definition}
Two families $(E_i)_{i\in I}$ and $(F_i)_{i\in I}$, with common indexed set $I$, are called combinatorially equivalent if one has \[\bigcap_{i\in J} E_i \ne \emptyset \iff \bigcap_{i\in J} F_i \ne \emptyset\] for every subset $J\subset I$.
\end{definition}
\begin{remark}
If $\alpha$ and $\beta$ are families of subsets of a set $K$ that are combinatorially equivalent, then one has $\text{ord}(\alpha)=\text{ord}(\beta)$.
\end{remark}
We need the following lemma (see \cite{coornaert2015topological}) to prove the measurability of $\mathcal{D} \big( \alpha_0^{n-1} \left( \omega \middle\vert \mathcal{E} \right)\big)$.

\begin{lemma}\label{lemma1}
Let $(F_i)_{i\in I}$ be a finite family of closed subsets of $K$ and $(U_i)_{i\in I}$ a family of open subsets of $K$ such that $F_i \subset U_i$ for all $i\in I$. Then there exists a family $(V_i)_{i\in I}$ of open subsets of $K$ satisfying the following conditions:
\begin{enumerate}
\item[(i)]
one has $F_i \subset V_i \subset \overline{V_i} \subset U_i$ for all $i\in I$;
\item[(ii)]
the families $(F_i)_{i\in I}$, $(V_i)_{i\in I}$ and $(\overline{V_i})_{i\in I}$ are combinatorially equivalent.
\end{enumerate}
\end{lemma}

\begin{theorem}
For every $n\in \mathbb{N}$ and finite open cover $\alpha$ of $X$, $\omega\mapsto \mathcal{D}\big(\alpha^{n-1}_0\left( \omega \middle\vert \mathcal{E} \right)\big)$ is measurable.
\end{theorem}

\begin{proof}
We can assume $\#\alpha=l$, $J={\{1,2,\dots,l\}}^{n}$. Define $f:\mathscr{P}_{k}^{l^n+1}\to \{-1\}\cup\mathbb{N}$ by \[f\big(K,(F^j)_{j\in J}\big)=-1+\min_{\beta} \sup_{x\in K} \sum_{B\in \beta} 1_{B}(x)\] where $\beta$ runs over all families of open subsets of $K$ such that \[\beta \succ (K\backslash{F^j})_{j \in J} \text{ and } K=\bigcup_{B\in \beta}B.\]
For $K=\emptyset$ or $\beta \in \emptyset$ we set $f=-1$. Obviously,
\[
\mathcal{D}\big(\alpha^{n-1}_0\left( \omega \middle\vert \mathcal{E} \right)\big)=f\big(\mathcal{E}_{\omega},(\mathcal{E}_{\omega}\backslash A^{j}_{\mathcal{E}}(\omega))_{j\in J}\big)\geq 0.
\]
Note that $\omega\mapsto\mathcal{E}_\omega$ is measurable and for every $j\in J$, $\{(\omega,x):x\in \mathcal{E} \backslash A^{j}_\mathcal{E}(\omega)\}=\mathcal{E}\backslash A^{j}_{\mathcal{E}}$, the graph of multifunction $\omega \mapsto \mathcal{E}_\omega \backslash A^{j}_{\mathcal{E}}(\omega)$ is measurable, so it suffices to prove that \[ \mathscr{Q}_q = \left\{ \big(K,(F^j)_{j\in J}\big):0\leq f(K,(F^j)_{j\in J}) \leq q \right\} \] is measurable with respect to the product $\sigma$-algebra $\otimes_{i=0}^{l^n} \sigma(\mathscr{T}_H)$ for any $q\in \mathbb{N}$, where $\mathscr{T}_H$ is the Hausdorff topology on $\mathscr{P}_k$ and $\sigma(\mathscr{T}_H)$ is the Borel $\sigma$-algebra generated by $\mathscr{T}_H$. In fact, $\mathscr{Q}_q$ is open in $\mathscr{P}_k^{l^n+1}$.

For any $(K_0,(F^j_0)_{j\in J})\in \mathscr{Q}_q$, there exists a family of finite closed subsets $(E^i)_{i\in I}$ of $K_0$ and a map $\phi : I \to J$ such that
\begin{equation*}
\begin{cases} K_0 = \bigcup_{i\in I} E^i, \\
E^i \subset K_0 \backslash F_0^{\phi(i)}, & \forall i\in I, \\
\sum_{i\in I} 1_{E^i}(x) \leq q+1, & \forall x\in K_0.
\end{cases}
\end{equation*}
Applying Lemma \ref{lemma1} on $E^i \subset K_0 \backslash F_0^{\phi(i)}$, $i\in I$, there exists a family of open subsets $(V^i)_{i\in I}$ of $K_0$ such that
\[ E^i \subset V^i \subset \overline{V^i} \subset K_0 \backslash F_0^{\phi(i)}, \forall i\in I\]
and $(E^i)_{i\in I}$, $(V^i)_{i\in I}$ and $(\overline{V^i})_{i\in I}$ are combinatorially equivalent. We have
\[ \sup_{x\in K_0} \sum_{i\in I} 1_{\overline{V^i}}(x) = \sup_{x\in K_0} \sum_{i\in I} 1_{V^i}(x) = \sup_{x\in K_0} \sum_{i\in I} 1_{E^i}(x) \leq q+1. \]

Set
\[ \mathscr{K}_1 = \left\{ K\in \mathscr{P}_k : K \subset \bigcup_{i\in I} V^i \right\} \]
\[ \mathscr{K}_2 = \bigcap_{\substack{\Gamma\subset I \\ \#\Gamma > q+1}} \left\{ K\in \mathscr{P}_k : K \cap \bigcap_{i\in \Gamma} \overline{V^i} = \emptyset \right\} \]
and for $j\in J$ set
\[ \mathscr{F}_j = \left\{ F\in \mathscr{P}_k : F\subset \bigcap_{i\in\phi^{-1}\{j\}} (X\backslash \overline{V^i}) \right\}. \]
Note that $\cup_{i\in I} V^i$ and $X\backslash \overline{V^i}$ are open in $X$ and $\cap_{i\in \Gamma} \overline{V^i}$ is closed in $X$, so (see \cite{kuratowski1948topologie}) $\mathscr{K}_1$, $\mathscr{K}_2$ and $\mathscr{F}_j$ are open subsets of $\mathscr{P}_k$. We have \[ \mathscr{O} = (\mathscr{K}_1 \cap \mathscr{K}_2) \times \prod_{j\in J} \mathscr{F}_j\] is open in $\mathscr{P}_k^{l^n+1}$ such that $(K_0,(F^j_0)_{j\in J})\in \mathscr{O} \subset\mathscr{Q}_q$. Indeed, for any $(K,(F^j)_{j\in J})\in \mathscr{O}$, by the construction of $\mathscr{K}_1$ and $\mathscr{F}_j$, $j\in J$, $(\overline{V^i})_{i\in I}$ is always a closed cover of $K$ that are finer than $(K\backslash F^j)_{j\in J}$ and satisfy \[ \sup_{x\in K} \sum_{i\in I} 1_{\overline{V^i}}(x) \leq q+1\] by the construction of $\mathscr{K}_2$. Since  $(K_0,(F^j_0)_{j\in J})$ is arbitrary, $\mathscr{Q}_q$ is open in $\mathscr{P}_k$.

To prove that the open set $\mathscr{Q}_q$ is measurable, we simply observe that $(X,d)$ is a compact metric space, so (see \cite{kuratowski1961topologie,michael1951topologies})
$\mathscr{P}_k$ is a compact metric space as well with the Hausdorff distance ($\emptyset$ is an isolated point in $\mathscr{P}_k$). It follows that $\mathscr{P}_k$ is second countable and so is $\mathscr{P}_k^{l^n+1}$. We have (see \cite{dudley2002real})
\[ \bigotimes_{i=0}^{l^n} \sigma (\mathscr{T}_H) = \sigma \Big(\prod_{i=0}^{l^n} \mathscr{T}_H\Big). \]
Thus the open set $\mathscr{Q}_q$ is $\otimes_{i=0}^{l^n} \sigma (\mathscr{T}_H)$ measurable and the proof of the theorem is complete.
\end{proof}

\begin{corollary}\label{corol1}
For every $n\in \mathbb{N}$ and finite open cover $\alpha$ of $X$, $\mathcal{D}\big(\alpha_{0}^{n-1} \left( \omega \middle\vert \mathcal{E} \right)\big)$ is integrable.
\end{corollary}

\begin{proof}
We can assume $\#\alpha=l$, then $\#\alpha_0^{n-1}\left( \omega \middle\vert \mathcal{E} \right) \leq l^n$. We have \[ D\big(\alpha_{0}^{n-1} \left( \omega \middle\vert \mathcal{E} \right)\big) \leq \text{ord}\big(\alpha_{0}^{n-1} \left( \omega \middle\vert \mathcal{E} \right)\big) \leq l^n - 1 \implies \mathbb{E} D\big(\alpha_{0}^{n-1} \left( \omega \middle\vert \mathcal{E} \right)\big) \leq l^n-1.\]
\end{proof}

\section{Mean topological dimension}\label{sec3}

Before introducing the mean topological dimension, we first review some propositions of $\mathcal{D}(\alpha)$. These propositions can be found in \cite{lindenstrauss2000mean} and \cite{coornaert2015topological}.

\begin{proposition}\label{subadditive}
Let $X$ be a normal space. Let $\alpha$ and $\beta$ be finite open covers of $X$. Then one has
\[
\mathcal{D}(\alpha\vee\beta)\leq\mathcal{D}(\alpha)+\mathcal{D}(\beta).
\]
\end{proposition}

\begin{definition}
Let $X$ and $Y$ be topological spaces. Let $\alpha$ be a finite open cover of $X$. A continuous map $f:X\to Y$ is said to be $\alpha$-compatible if there exists a finite open cover $\beta$ of $Y$ such that $f^{-1}(\beta)\succ\alpha$. We will use the notation $f\succ\alpha$ to denote that $f$ is $\alpha$-compatible.
\end{definition}

\begin{proposition} \label{compatible}
If $X$ is compact, $f:X\to Y$ a continuous function such that for every $y\in Y$, $f^{-1}\{y\}$ is a subset of some $U\in\alpha$, then $f$ is $\alpha$-compatible.
\end{proposition}

\begin{proposition}\label{lemma3}
Let $X$ be a topological space and $\alpha$ a finite open cover of $X$. Suppose that there exist a topological space $Y$ and an $\alpha$-compatible continuous map $f:X\to Y$. Then one has
$\mathcal{D}(\alpha)\leq\text{dim}Y$.
\end{proposition}

Note that \begin{equation*} \begin{split}
\mathcal{D}\big(\alpha_{0}^{n+m-1} \left( \omega \middle\vert \mathcal{E} \right)\big)
& \leq \mathcal{D}\big(\alpha_{0}^{n-1} \left( \omega \middle\vert \mathcal{E} \right)\big) + \mathcal{D}\Big(\bigvee_{i=n}^{n+m-1} (T^i_\omega)^{-1} \alpha \left( \vartheta^i \omega \middle\vert \mathcal{E} \right)\Big) \\
& =  \mathcal{D}\big(\alpha_{0}^{n-1} \left( \omega \middle\vert \mathcal{E} \right)\big) + \mathcal{D}\Big( (T^n_\omega)^{-1} \bigvee_{i=0}^{m-1} (T^i_\omega)^{-1} \alpha \left( \vartheta^{n+i} \omega \middle\vert \mathcal{E} \right)\Big) \\
& \leq \mathcal{D}\big(\alpha_{0}^{n-1} \left( \omega \middle\vert \mathcal{E} \right)\big) + \mathcal{D}\Big(\bigvee_{i=0}^{m-1} (T^i_\omega)^{-1} \alpha \left( \vartheta^{n+i} \omega \middle\vert \mathcal{E} \right)\Big).
\end{split}
\end{equation*}
Set $\mathcal{D}(\alpha_{0}^{n-1} \left( \omega \middle\vert \mathcal{E} \right)) = q_n(\omega)$, then \[ q_{n+m}(\omega) \leq q_n(\omega) + q_m(\vartheta^n \omega).\] By Corollary \ref{corol1} and Kingman's subadditive ergodic theorem (see \cite{arnold2013random}), there exists an integrable function $q:\Omega \to \bar{\mathbb{R}}$ such that for $\mathbb{P}$-a.a. $\omega$, \[ \frac{1}{n} q_n(\omega) \to q(\omega)\] and satisfy
\[ \mathbb{E}q = \lim_{n\to\infty} \frac{1}{n} \mathbb{E}q_n = \inf_{n\in {\mathbb{N}^{+}}} \frac{1}{n} \mathbb{E}q_n < \infty.\]
So we can define the mean topological dimension for a bundle RDS $T$ with respect to an open cover $\alpha$ of $X$ as
\[
\text{mdim} (T,\alpha) = \lim_{n\to\infty} \mathbb{E} \mathcal{D}\big(\alpha_{0}^{n-1} \left( \omega \middle\vert \mathcal{E} \right)\big)
\]
\begin{remark} $\text{mdim} (T,\alpha) \geq 0$. \end{remark}
\begin{remark} \label{rek1}$\alpha\succ\beta\implies\text{mdim}(T,\alpha)\geq\text{mdim}(T,\beta)$. \end{remark}
\begin{remark} $\text{mdim}(T,\alpha) \leq \mathbb{E} \mathcal{D}\big(\alpha \left( \omega \middle\vert \mathcal{E} \right)\big)$. \end{remark}
\begin{proof}
By Kingman's subadditive ergodic theorem and Proposition \eqref{subadditive} we have
\begin{equation*} \begin{split}
\text{mdim}(T,\alpha) & = \inf_{n\in\mathbb{N}^+} \frac{1}{n} \mathbb{E} \mathcal{D}\big(\alpha_{0}^{n-1} \left( \omega \middle\vert \mathcal{E} \right)\big) \\
& \leq \frac{1}{n} \mathbb{E} \mathcal{D}\Big(\bigvee_{i=0}^{n-1} (T^i_\omega)^{-1} \alpha \left( \vartheta^{i} \omega \middle\vert \mathcal{E} \right)\Big) \\
& \leq \frac{1}{n} \mathbb{E} \sum_{i=0}^{n-1} \mathcal{D}\Big((T^i_\omega)^{-1} \alpha \left( \vartheta^{i} \omega \middle\vert \mathcal{E} \right)\Big) \\
& \leq \frac{1}{n} \sum_{i=0}^{n-1} \mathbb{E} \mathcal{D}\Big( \alpha \left( \vartheta^{i} \omega \middle\vert \mathcal{E} \right)\Big) \\
& = \frac{1}{n} \sum_{i=0}^{n-1} \mathbb{E} \mathcal{D}\big( \alpha \left( \omega \middle\vert \mathcal{E} \right)\big) \\
& = \mathbb{E} \mathcal{D}\big( \alpha \left( \omega \middle\vert \mathcal{E} \right)\big).
\end{split}
\end{equation*}
\end{proof}
\begin{definition}
The mean topological dimension of bundle RDS $T$, denoted by $\text{mdim}(T)$, is defined by \[\text{mdim}(T)=\sup_{\alpha} \text{mdim}(T,\alpha),\] where $\alpha$ runs over all finite open covers of $X$.
\end{definition}
\begin{remark} $\text{mdim}(T)\geq 0$.\end{remark}
\begin{remark} By the compactness of $X$ and every subcover is a refinement, $\alpha$ could run over all open covers (not necessarily finite) of $X$ in the definition of mdim. \end{remark}
\begin{remark} By the compactness of $X$, one can replace $\sup_\alpha$ by $\sup_k$ and pick up any sequence $\alpha(k)$ of open covers of $X$ such that the diameter of $\alpha(k)$ tends to zero as $k\to\infty$. \end{remark}
\begin{proposition}
If for $\mathbb{P}$-a.a. $\omega$, $\mathcal{E}_\omega \subset \mathcal{E}_{\theta\omega}$ and $T_\omega$ is an inclusion mapping, namely, $\Theta(\omega,x)=(\vartheta\omega,x)$, then $\text{mdim}(T)=0$.
\end{proposition}
\begin{proof}
Suppose $\alpha$ is any open cover of $X$ and $\#\alpha=l$. We assume that $T_\omega$ is inclusion on a set $A\subset\Omega$ with full measure. Since $\vartheta$ is a $\mathbb{P}$-preserving transformation, $\{\omega:\theta\omega\notin A\}$ is a null set. For $\mathbb{P}$-a.a. $\omega$, we have
\begin{equation*} \begin{split}
\mathcal{D}\big(\alpha_{0}^{l-1} \left( \omega \middle\vert \mathcal{E} \right)\big) & = \mathcal{D}\big(\alpha_{0}^{l} \left( \omega \middle\vert \mathcal{E} \right)\big) = \mathcal{D}\big(\alpha_{0}^{l+1} \left( \omega \middle\vert \mathcal{E} \right)\big) = \dots \\
& = \mathcal{D}\big( \underbrace{ \alpha \left( \omega \middle\vert \mathcal{E} \right) \vee \alpha \left( \omega \middle\vert \mathcal{E} \right) \vee \dots \vee \alpha \left( \omega \middle\vert \mathcal{E} \right)}_{l} \big),
\end{split}
\end{equation*}
so \[ \text{mdim}(T,\alpha)=\lim_{n\to\infty} \frac{1}{n} \mathbb{E} \mathcal{D}\big(\alpha_{0}^{n-1} \left( \omega \middle\vert \mathcal{E} \right)\big) = 0 \implies \text{mdim}(T)=0. \]
\end{proof}
\begin{proposition}
If $X$ has finite topological dimension, then $\text{mdim}(T)=0$.
\end{proposition}
\begin{proof}
Let $\alpha$ be any open cover of $X$. Then \[ \mathcal{D}\big(\alpha_{0}^{n-1} \left( \omega \middle\vert \mathcal{E} \right)\big) \leq \text{dim} \mathcal{E}_\omega \leq \text{dim} X < \infty \] and hence $\text{mdim}(T)=0$.
\end{proof}
\begin{proposition}
If a random compact set $\mathcal{C}\subset\mathcal{E}$ is strictly forward invariant, then $\text{mdim}\left(T\middle\vert_\mathcal{C}\right) \leq \text{mdim}(T)$.
\end{proposition}
\begin{proof}
It suffices to proof that for any open cover $\alpha$ of $X$ and $n\in\mathbb{N}^+$,
\[ \mathcal{D}\big(\alpha_{0}^{n-1} \left( \omega \middle\vert \mathcal{C} \right)\big) \leq \mathcal{D}\big(\alpha_{0}^{n-1} \left( \omega \middle\vert \mathcal{E} \right)\big).\]
To prove this, we only need to show that for any open cover $\beta\succ\alpha^{n-1}_0 \left( \omega \middle\vert \mathcal{E} \right)$ of $\mathcal{E}_\omega$, there exists an open cover $\gamma$ of $\mathcal{C}_\omega$ that refines $\alpha^{n-1}_0 \left( \omega \middle\vert \mathcal{C} \right)$ such that \[
\sup_{x\in\mathcal{C}_\omega} \sum_{V\in\gamma} 1_V (x) \leq \sup_{x\in\mathcal{E}_\omega} \sum_{U\in\beta} 1_U (x). \]Indeed, we can take $\gamma = \{ U\cap\mathcal{C}_\omega : U\in\beta\}$.
\end{proof}

For each $n\geq 1$ and a positive random variable $\epsilon=\epsilon (\omega)$ define a family of metrics $d^{\omega}_{\epsilon,n}$ on $\mathcal{E}_\omega$ by the formula
\[
d^{\omega}_{\epsilon,n}(x,y)=\max_{0\leq k<n} \frac{d(T^k_\omega x,T^k_\omega y)}{\epsilon({\vartheta}^{k}\omega)}, \qquad x,y\in\mathcal{E}_\omega.
\]

Let $\alpha = \{ A^{(i)}:i=1,\dots,l\}$ be a finite open cover of $X$.
Define the mesh of $\alpha\left(\omega\middle\vert\mathcal{E}\right)$ according to the metric $d^\omega_{\epsilon,n}$ by
\[
\text{diam} \big( \alpha\left(\omega\middle\vert\mathcal{E}\right), d^\omega_{\epsilon,n}\big) = \max_{1\leq i \leq l} \text{diam} \big(A^{(i)}_{\mathcal{E}}(\omega),d^\omega_{\epsilon,n}\big)
\]
where
\[
\text{diam} \big(A^{(i)}_{\mathcal{E}}(\omega),d^\omega_{\epsilon,n}\big) = \sup\{ d^\omega_{\epsilon,n} (x,y) : x,y\in A^{(i)}_{\mathcal{E}}(\omega)\}.
\]
Fix $\omega$, let $\text{cov}(\omega,\epsilon,n)$ be the minimal cardinality of a covering of $\mathcal{E}_\omega$ by sets of $d^\omega_{\epsilon,n}$-diameter less than 1, that is
\[
\text{cov}(\omega,\epsilon,n) = \inf \{ \# \alpha : \text{diam} \big( \alpha\left(\omega\middle\vert\mathcal{E}\right), d^\omega_{\epsilon,n}\big) < 1 \}.
\]

One familiar with the classic topological entropy theory in the deterministic case will have no difficulty in extending separated sets to the random case. We define
\[
\text{sep}(\omega,\epsilon,n) = \sup \{ \# F : \forall x,y\in F\subset \mathcal{E}_\omega , x \ne y \implies d^\omega_{\epsilon,n} (x,y) \geq 1 \}
\]
as the maximum cardinality of $(\omega,\epsilon,n)$-separated sets.
As in the deterministic case (see \cite{brin2002introduction}), we have
\begin{equation}\label{ineq1}
\text{sep}(\omega,2\epsilon,n) \leq \text{cov}(\omega,2\epsilon,n) \leq \text{sep}(\omega,\epsilon,n).
\end{equation}

Set \[
S(\omega,\epsilon) = \lim_{n\to\infty} \frac{1}{n} \log \text{cov} (\omega,\epsilon,n).
\]
Noticed that $\log\text{cov}(\omega,\epsilon,n)$ is a subadditive function of $n$ and the limit above can be replaced by an infimum.
Obviously $S$ is monotone nondecreasing as $\epsilon\to 0$, and for the purpose of measuring how fast it increase, we define the metric mean dimension of $T$ on the fiber $\mathcal{E}_\omega$ as
\[
\text{mdim}_\text{M} (T,\omega) = \liminf_{1 > \epsilon\to 0} \frac{S(\omega,\epsilon)}{-\log \epsilon}.
\]
It follows from \eqref{ineq1} and $\lim_{\epsilon\to 0} \frac{\log\epsilon}{\log2\epsilon} = 1$ that
\begin{align*}
\text{mdim}_\text{M} (T,\omega) & = \liminf_{1>\epsilon\to0} \frac{1}{-\log\epsilon} \limsup_{n\to\infty} \frac{1}{n} \log \text{sep} (\omega,\epsilon,n) \\
& = \liminf_{1>\epsilon\to0} \frac{1}{-\log\epsilon} \liminf_{n\to\infty} \frac{1}{n} \log \text{sep} (\omega,\epsilon,n).
\end{align*}

Since $\text{sep}(\omega,\epsilon,n)$ is measurable (see \cite{kifer2001topological}), it follows that $\text{mdim}_\text{M}(T)(\omega)$ is measurable as well. We have the following definition.
\begin{definition}
The metric mean topological dimension of bundle RDS $T$, denoted by $\text{mdim}_\text{M}(T)$, is defined by
\[
\text{mdim}_\text{M}(T)=\mathbb{E} \text{mdim}_\text{M}(T,\omega).
\]
\end{definition}

Set \[
S'(\omega,\epsilon) = \limsup_{n\to\infty} \frac{1}{n} \log \text{sep} (\omega,\epsilon,n)\]
and compare
\[
\text{mdim}_\text{M}(T) = \mathbb{E} \liminf_{1>\epsilon\to 0} \frac{S'(\omega,\epsilon)}{-\log\epsilon}
\]
with (see \cite{kifer2001topological})
\[
h_\text{top}(T) = \mathbb{E} \lim_{\epsilon\to 0} S'(\omega,\epsilon).
\]
We immediately conclude the relationship between the metric mean dimension and the topological entropy of a system.
\begin{proposition}
If $\text{mdim}_\text{M}(T) \ne 0$, then $h_\text{top}(T) = \infty$.
\end{proposition}
\begin{proof}
If $\text{mdim}_\text{M}(T) > 0$, then there exists a set $A\subset \Omega$ with positive measure such that $\text{mdim}_\text{M}(T)(\omega) > 0$ on $A$. Let
\[
A_k = \left\{ \omega\in\Omega : \liminf_{1>\epsilon\to 0} \frac{S'(\omega,\epsilon)}{-\log\epsilon} \geq \frac{1}{k} \right\}
\]
and we have $\cup_{k=1}^{\infty} A_k = A$, so there exists an integer $m$ such that $\mathbb{P}(A_m)>0$ and for any $\omega\in A_m$,
\[
\lim_{\epsilon\to 0} S'(\omega,\epsilon) = \infty.
\]
So
\[
h_\text{top}(T) = \mathbb{E} \lim_{\epsilon\to 0} S'(\omega,\epsilon) \geq \int_{A_m} \lim_{\epsilon\to 0} S'(\omega,\epsilon) \mathrm{d}\mathbb{P} = \infty.
\]
\end{proof}
\begin{corollary}
If $T$ has finite topological entropy, then $\text{mdim}_\text{M}(T) = 0$.
\end{corollary}

We conclude this section by pointing out that the mean topological dimension is not larger than the metric mean dimension for bundle RDS as well as in the deterministic case. This result will be proved if we can show that for any finite open cover $\alpha$ of $X$ and $\mathbb{P}$-a.a. $\omega$,
\[
\lim_{n\to\infty} \frac{1}{n} \mathcal{D} \big( \alpha_0^{n-1} \left( \omega \middle\vert \mathcal{E} \right) \big) \leq \text{mdim}_\text{M} (T,\omega).
\]
The proof of this theorem is almost identical with its deterministic version (see \cite{lindenstrauss2000mean}), the major change being the substitution of $T_\omega$ for $T$, and we omit the proof.

\begin{theorem}
For any bundle RDS $T$,
\[
\text{mdim}(T) \leq \text{mdim}_\text{M}(T).
\]
\end{theorem}
\begin{corollary}
If $T$ has finite topological entropy, then $\text{mdim}(T) = 0$.
\end{corollary}

\section{The small boundary property}\label{sec4}

\begin{definition}
Let $T$ be a bundle RDS, and $E$ a Borel subset of $X$. We define the $\omega$-orbit capacity of the set $E$ to be \[
\text{ocap}^\omega (E) = \lim_{n\to\infty} \frac{1}{n} \sup_{x\in\mathcal{E}_\omega} \sum_{i=0}^{n-1} 1_E (T^i_\omega x).
\]
A set $E\subset X$ will be called $T$-small (or simply small) if $\text{ocap}^\omega (E)=0, \mathbb{P}$-a.s.
\end{definition}
Let $b_n(\omega)=\sup\{\sum_{i=0}^{n-1} 1_E(T^i_\omega x):x\in\mathcal{E}_\omega\}$. We remark that the limit above exists for a.a. $\omega$ by Kingman's subadditive ergodic theorem since $b_n(\omega)\in\mathbb{L}^1(\Omega,\mathcal{F},\mathbb{P})$ and $b_{n+m}(\omega) \leq b_n(\omega) + b_m(\vartheta^n \omega)$.
So then for any Borel set $E\subset X$, $\omega\mapsto \text{ocap}^\omega (E)$ is measurable.

\begin{definition}
A bundle RDS $T$ has the small boundary property (SBP) if every point $x\in X$ and every open subset $U$ that contains $x$ there is a neighborhood $V \subset U$ of $x$ with small boundary.
\end{definition}

Let $\pi_\Omega$ be the canonical projection from $\Omega\times X$ onto $\Omega$.
Let $\mathcal{P}_{\mathbb{P}}(\Omega\times X)$ be the space of probability measures on $\Omega\times X$ having the marginal $\mathbb{P}$ on $\Omega$
and set $\mathcal{P}_{\mathbb{P}}(\mathcal{E})=\{\mu \in \mathcal{P}_{\mathbb{P}}(\Omega\times X): \mu(\mathcal{E})=1\}$.
Denote by $\mathcal{I}_{\mathbb{P}}(\mathcal{E})$ the space of all $\Theta$-invariant measures in $\mathcal{P}_{\mathbb{P}}(\mathcal{E})$.
Any $\mu \in \mathcal{I}_{\mathbb{P}}(\mathcal{E})$ on $\mathcal{E}$ disintegrates $\mathrm{d}\mu(\omega,x)=\mathrm{d}\mu_\omega(x)\mathrm{d}\mathbb{P}(\omega)$ (see \cite{dudley2002real}),
where $\mu_\omega$ are regular conditional probabilities with respect to the sub $\sigma$-algebra $\mathcal{F}_\mathcal{E}$ formed by all sets $(A\times X)\cap \mathcal{E}$ with $A\in \mathcal{F}$.
This means that $\mu_\omega$ is a probability measure on $\mathcal{E}_\omega$ for $\mathbb{P}$-a.a. $\omega$ and for any measurable set $R\subset \mathcal{E}$, $\mathbb{P}$-a.s.
$\mu_\omega(R_\omega)=\mu \left( R \middle\vert \mathcal{F}_\mathcal{E} \right)$, where $R_\omega = \{x:(\omega,x)\in R \}$, and so $\mu(R) = \int \mu_\omega(R_\omega) \mathrm{d}\mathbb{P}(\omega)$.
\begin{proposition}
For closed $E\subset X$, we have \[\mathbb{E} \text{ocap}^\omega (E) = \sup \{\mu (\Omega\times E) : \mu\in\mathcal{I}_\mathbb{P}(\mathcal{E})\}.\]
\end{proposition}
\begin{proof}
Let $\mu\in\mathcal{I}_\mathbb{P}(\mathcal{E})$, for any $\epsilon > 0$ and large enough $n$,
\begin{equation*} \begin{split}
\mu (\Omega\times E) & = \int 1_{\Omega\times E}(\omega,x) \mathrm{d}\mu \\
& = \frac{1}{n} \int \sum_{i=0}^{n-1} 1_{\Omega\times E}(\Theta^i (\omega,x))\mathrm{d}\mu \\
& = \frac{1}{n} \iint \sum_{i=0}^{n-1} 1_{\Omega\times E}(\Theta^i (\omega,x))\mathrm{d}\mu_\omega \mathrm{d}\mathbb{P} \\
& = \mathbb{E} \Big( \frac{1}{n} \int \sum_{i=0}^{n-1} 1_E (T^i_\omega x) \mathrm{d}\mu_\omega \Big) \\
& < \mathbb{E} \text{ocap}^\omega (E) + \epsilon
\end{split}
\end{equation*}
Since $\epsilon$ is arbitrary, $\mu (\Omega\times E) \leq \mathbb{E} \text{ocap}^\omega (E)$;

Conversely, for any $\epsilon>0$, consider the map $\Gamma_{\epsilon,n} : \Omega \to \mathscr{P}_k$ defined as \[ \Gamma_n (\omega) = \left\{ x\in\mathcal{E}_\omega : \frac{1}{n} \sum_{i=0}^{n-1} 1_E (T^i_\omega x) > \text{ocap}^\omega (E) - \epsilon \right\}. \] Note that if there exists $x\in \mathcal{E}_\omega$ such that $x\notin \Gamma_n(\omega)$, then there is an open neighborhood $U$ of $x$ such that for all $y\in U$, $y\notin \Gamma_n (\omega)$. To see this consider that among the points $x, T_\omega x, \dots, T^{n-1}_\omega x$ which misses the closed set $E$, denoted by $T^i_\omega x, i\in I'$. For any $i\in I'$, since $T_\omega$ is continuous there is some open neighborhood $U_i$ of $x$ such that $T^i_\omega U_i$ misses $E$ as well. Then $U=\cap_{i\in I'} U_i$ will satisfy our requirements, so then $\Gamma_n$ is closed (compact) valued.

We assume that $\text{ocap}^\omega (E)$ exists and be finite on a set $A\subset \Omega$ with full measure. Observe that the graph of $\Gamma_n$, \[
G_n = \left\{ (\omega,x)\in\mathcal{E} : \frac{1}{n} \sum_{i=0}^{n-1} 1_E (T^i_\omega x) > \text{ocap}^{\omega} (E) - \epsilon  \right\} \in \mathcal{F}\otimes\mathcal{B},
\]
so $\pi_\Omega G_n \in \mathcal{F}$ and \[
\bigcup_{n=1}^{\infty} \bigcap_{m\geq n} \pi_\Omega G_m = A.
\]
Set $\cap_{m\geq n} \pi_\Omega G_m = A_{n}$. Applying selection theorem (see \cite{castaing2006convex}) on $\Gamma_n \vert_{A_{n}}$, there exists an $\mathcal{F}\cap A_{n}, \mathcal{B}$ measurable function $\gamma_{n} : A_{n}\to X$, $\gamma_{n}(\omega)\in\Gamma_n (\omega)\subset \mathcal{E}_\omega$ such that \[
\frac{1}{n} \sum_{i=0}^{n-1} 1_E (T^i_\omega \gamma_{n}(\omega)) > \text{ocap}^\omega (E) - \epsilon, \qquad \omega \in A_{n}.
\]
For $\omega \notin A_{n}$, applying selection theorem again on the map $\omega\mapsto\mathcal{E}_\omega$, there exists an measurable function $e(\omega)\in\mathcal{E}_\omega$. Now we can define probability measures $\nu^{(n)}$ on $\mathcal{E}$ via their measurable disintegrations \[
\nu_{\omega}^{(n)} = \delta_{\zeta_n(\omega)} \text{ where } \zeta_n (\omega) = \begin{cases}
{\gamma_n (\omega)}, & \omega \in A_{n} \\
{e (\omega)}, & \omega \notin A_{n},
\end{cases}
\]
so that $ \mathrm{d} \nu^{(n)} (\omega,x) = \mathrm{d}\nu^{(n)}_\omega (x) \mathrm{d}\mathbb{P}(\omega)$
and set
\[
\mu^{(n)} = \frac{1}{n} \sum_{i=0}^{n-1} \nu^{(n)} \circ \Theta^{-i}.
\]
Then any limit point of $\mu_{n}$ for $n\to\infty$ in the topology of weak convergence is in $\mathcal{I}_\mathbb{P} (\mathcal{E})$. In fact, suppose $\mu_{n}\to \mu \in \mathcal{P}_\mathbb{P} (\mathcal{E})$,then for any $f\in \mathbb{L}^1 (\Omega,C(X))$, we have
\begin{equation*}\begin{split}
\bigg\vert \int f\circ\Theta \mathrm{d}\mu - \int f\mathrm{d}\mu \bigg\vert & = \lim_{n\to\infty} \bigg\vert \int f\circ\Theta\mathrm{d} \mu^{(n)} - \int f\circ\mathrm{d}\mu^{(n)}\bigg\vert \\
& = \lim_{n\to\infty} \frac{1}{n} \bigg\vert \int (f\circ\Theta^n - f)\mathrm{d}\nu^{(n)} \bigg\vert \\
& \leq \lim_{n\to\infty} \frac{2}{n} \Vert f \Vert = 0.
\end{split}
\end{equation*}
So $\mu\in\mathcal{I}_{\mathbb{P}}(\mathcal{E})$. Since for any $0\leq i < n$,
\begin{align*}
\nu^{(n)} \big( \Theta^{-i} (\Omega\times E) \big)
& = \int 1_{\Omega\times E} \big( \Theta^i (\omega,x) \big) \mathrm{d}\nu_\omega^{(n)}(x)\mathrm{d}\mathbb{P}(\omega) \\
& = \int_{\Omega} 1_E \big( T^i_\omega \zeta_n (\omega) \big) \mathrm{d}\mathbb{P}(\omega),
\end{align*}
we see that when $n$ is large enough,
\begin{equation*}\begin{split}
\mu^{(n)}(\Omega\times E)
& = \frac{1}{n} \sum_{i=0}^{n-1} \nu^{(n)} \big( \Theta^{-i}(\Omega\times E) \big) \\
& = \int_{\Omega} \frac{1}{n} \sum_{i=0}^{n-1} 1_E \big(T^i_\omega \zeta_n (\omega)\big) \mathrm{d}\mathbb{P}(\omega) \\
& \geq \int_{A_{n}} \frac{1}{n} \sum_{i=0}^{n-1} 1_E \big(T^i_\omega \gamma_n (\omega)\big) \mathrm{d}\mathbb{P}(\omega) \\
& > \int_{A_{n}} \text{ocap}^{\omega}(E) \mathrm{d}\mathbb{P}(\omega) - \epsilon \mathbb{P} (A_n) \\
& \geq \int_{\Omega} \text{ocap}^{\omega}(E) \mathrm{d}\mathbb{P}(\omega) - \epsilon - \int_{\Omega\backslash A_{n}} \text{ocap}^{\omega}(E) \mathrm{d}\mathbb{P}(\omega).
\end{split}
\end{equation*}
Since $\epsilon$ is arbitrary and $\text{ocap}^{\omega}(E)$ is integrable, by the absolute continuity of integral we have \[\mathbb{E} \text{ocap}^\omega (E) = \sup \{\mu (\Omega\times E) : \mu\in\mathcal{I}_\mathbb{P}(\mathcal{E})\}.\]
\end{proof}

\begin{corollary}
For closed set $E\subset X$, $E$ is $T$-small $\iff \Omega\times E$ is $\mu$-null set for all $\mu \in \mathcal{P}_\mathbb{P}(\mathcal{E})$.
\end{corollary}

Recall that ergodic measures for RDS are ergodic measures for the skew product transformations and the ergodicity of the base system $(\Omega,\mathcal{F},\mathbb{P},\vartheta)$ is a necessary condition for the existence of an ergodic invariant measure for $\Theta$ (see \cite{crauel2003random}).

\begin{corollary}
If $\mathbb{P}$ is $\vartheta$-ergodic and $T$ is uniquely ergodic, then $T$ has the SBP.
\end{corollary}
\begin{proof}
Suppose $\mathcal{P}_\mathbb{P}(\mathcal{E}) = \{\mu\}$, then for any $x\in X$, the set
\[
\left\{ r>0 : \{y:d(x,y)=r\} \text{ is not small}\right\} = \left\{ r>0 : \mu\{y:d(x,y)=r\}\ne 0 \right\}
\]
is at most countable, and so there are many arbitrary small balls around $x$ with small boundary, and $T$ has the SBP.
\end{proof}

According to the definition of the SBP, as a consequence the following result relate the orbit capacity to a partitions of unity subordinate to an open cover $\alpha$ and allow us to define an $\alpha^{n-1}_0 \left(\omega\middle\vert\mathcal{E}\right)$-compatible map in the next theorem.

\begin{proposition}
If $T$ has the SBP, then for every finite open cover $\alpha$ of $X$, every $\epsilon>0$ and $\mathbb{P}$-a.a. $\omega$ there exists $N_0 (\omega,\epsilon)>0$ such that for any $N>N_0 (\omega,\epsilon)$, there is a subordinate partition of unity $\phi_j^\omega : X\to [0,1]$ with respect to $\alpha$ such that
\begin{equation}\label{n1}
\frac{1}{N}\sum_{i=0}^{N-1} 1_{A_\omega} (T^i_\omega x) < \epsilon, \text{ where } A_\omega = \bigcup_{j=1}^{\#\alpha} (\phi^\omega_j)^{-1} (0,1)
\end{equation}
\end{proposition}
\begin{proof}
Let $\alpha = (U_j)$. First find (using the SBP) a cover of $X$ by open sets with small boundary that refines $\alpha$, and by taking unions of these sets it is possible to find a refinement $\alpha'\succ\alpha$ such that there is a one-to-one correspondence between the elements $U_j$ of $\alpha$ and $U_j {'}$ of $\alpha {'}$ with $U_j {'}\subset U_j$.
Since $U_j '$ has small boundary, for $\mathbb{P}$-a.a. $\omega$, $\epsilon>0$ and big enough $N>N_0 (\omega,\epsilon)$, we have
\begin{equation}\label{n2}
\frac{1}{N} \sum_{i=0}^{N-1} 1_{\partial U_j '} (T^i_\omega x) < \frac{\epsilon}{\#\alpha}, \qquad x\in\mathcal{E}_\omega.
\end{equation}
Fixed $\omega$ and $x_0 \in \mathcal{E}_\omega$ which satisfied \eqref{n2}, consider that among the points $x_0,$ $T_\omega x_0,$ $\dots,$ $T^{n-1}_\omega x_0$ which misses the closed set $\partial U_j '$, denoted by $T^i_\omega x_0, i\in I'$. There exists a $\delta_0 = \delta(x_0) > 0$ such that $T^i_\omega x_0, i\in I'$ will miss $B(\partial U_j ',\delta_0)$ as well, where $B(F,r) = \{ x : d(x,F)<r\}$. By taking smaller $\delta_0$, one can replace the open ball $B(\partial U_j ',\delta_0)$ by closed ball $\overline{B}(\partial U_j ',\delta_0)$. So we have
\[
\frac{1}{N} \sum_{i=0}^{N-1} 1_{\overline{B}(\partial U_j ',\delta_0)} (T^i_\omega x_0) < \frac{\epsilon}{\#\alpha}.
\]
Again consider that among the points $x_0, T_\omega x_0, \dots, T^{n-1}_\omega x_0$ which misses the closed set $\overline{B}(\partial U_j ',\delta_0)$, denoted by $T^i_\omega x_0, i\in I''$.
For any $i\in I''$, since $T^i_\omega$ is continuous there is some open neighborhood $N_i (x_0)$ such that $T^i_\omega N_i (x_0)\cap \overline{B}(\partial U_j ',\delta_0) = \emptyset, i\in I''$.
Let $N(x_0)=\cap_{i\in I''} N_i (x_0)$. We have
\[
\frac{1}{N} \sum_{i=0}^{N-1} 1_{B(\partial U_j ',\delta_0)} (T^i_\omega y) < \frac{\epsilon}{\#\alpha}, \qquad y\in N(x_0).
\]

Let $x_0$ run over $\mathcal{E}_\omega$, since $\mathcal{E}_\omega$ is covered by $\cup_{x_0 \in \mathcal{E}_\omega} N(x_0)$, by the compactness of $\mathcal{E}_\omega$ there are finite many $N(x_k)$ which could cover $\mathcal{E}_\omega$. Let $\delta = \min \{\delta(x_k)\}$, then
\[
\frac{1}{N} \sum_{i=0}^{N-1} 1_{B(\partial U_j ',\delta)} (T^i_\omega x) < \frac{\epsilon}{\#\alpha}, \qquad x\in \mathcal{E}_\omega.
\]
We also take $\delta$ small enough so that $B(\partial U_j ',\delta) \subset U_j$ for all $j$. Now take \[
\psi_j^\omega (x) = \begin{cases}
1, & x \in U_j ' \\
\max \{0,1-\frac{1}{\delta} d(x,\partial U_j ')\}, & x \notin U_j '.
\end{cases}
\]
We define the function $\phi_j^\omega (x)$ as follows:
\begin{equation*} \begin{split}
\phi_1^\omega (x) & = \psi_1^\omega (x) \\
\phi_2^\omega (x) & = \min \{ \psi_2^\omega (x), 1-\phi_1^\omega (x)\} \\
\phi_3^\omega (x) & = \min \{ \psi_3 (x)^\omega, 1-\psi_1^\omega (x)-\psi_2^\omega (x)\} \\
\vdots &
\end{split}
\end{equation*}
These clearly satisfy the required conditions.
\end{proof}

\begin{theorem}
If $T$ has the SBP, then $\text{mdim}(T)=0$.
\end{theorem}

\begin{proof}
Take $\alpha$ a finite open cover of $X$, $\#\alpha=k$, $\epsilon>0$. Construct for $\mathbb{P}$-a.a. $\omega$, according to the above proposition, an $\alpha$ subordinate partition of unity which obeys \eqref{n1}.

Define $\Phi^\omega: X \to \mathbb{R}^k$ by \[x \mapsto \big(\phi_1^\omega (x), \phi_2^\omega (x), \dots, \phi_k^\omega (x)\big).\]

Define the map $f_N^\omega: X \to \mathbb{R}^{kN}$ by \[f_N^\omega (x) = \big( \Phi^\omega (x), \Phi^{\vartheta\omega} (T_\omega x), \dots, \Phi^{\vartheta^{N-1}\omega} (T^{N-1}_\omega x) \big).\]

We claim that $f_N^\omega (\mathcal{E}_\omega)$ is a subset of a finite number of $\epsilon k N$ dimensional affine subspaces of $\mathbb{R}^{kN}$.

Indeed, let $\{e^i_j : 1\leq i \leq N, 1\leq j \leq k\}$ be the standard base of $\mathbb{R}^{kN}$. Define for every index set $I\subset \{1,2,\dots,N\}$, $\# I < \epsilon N$, and every $\xi \in \{0,1\}^{kN}$
\[
C(I,\xi) = \text{span} \{ e^i_j : i\in I , 1\leq j \leq k\} + \xi.
\]
Then by \eqref{n1}, \[
f_N^\omega(\mathcal{E}_\omega)\subset \bigcup_{\substack{\# I < \epsilon N \\ \xi \in \{0,1\}^{kN}}} C(I,\xi).
\]
It is easy to see from Proposition \eqref{compatible} that $f_N^\omega$ is $\alpha^{N-1}_0 \left(\omega\middle\vert\mathcal{E}\right)$-compatible and by Proposition \eqref{lemma3} we see that \[
\mathcal{D}\big(\alpha^{N-1}_0 \left(\omega\middle\vert\mathcal{E}\right)\big) < \epsilon k N
\]
for $N>N_0 (\omega,\epsilon)$ and so $T$ has zero mean dimension.
\end{proof}


%

\end{document}